\documentclass[12pt]{article}

\usepackage{amssymb,amsmath,amsfonts}
\usepackage{amsthm}
\usepackage{enumitem}
\usepackage{color}
\usepackage{tikz}
\usepackage{booktabs,threeparttable}

\numberwithin{equation}{section} 
\theoremstyle{plain}
\newtheorem{theorem}{Theorem}[section]
\newtheorem{corollary}[theorem]{Corollary}

\newtheorem{lemma}[theorem]{Lemma}

\theoremstyle{definition}
\newtheorem{definition}{Definition}[section]

\theoremstyle{remark}

\allowdisplaybreaks

\begin{document}

\title{\bf The impact of the limit $q$-Durrmeyer operator on continuous functions }
\author{\"{O}vg\"{u} G\"{u}rel Y\i lmaz$^1$, Sofiya Ostrovska$^2$, Mehmet Turan$^2$}
\date{\today}
\maketitle

\begin{center}
{\it $^1$Recep Tayyip Erdogan University, Department of Mathematics, 53100, Rize, Turkey}\\
{\it $^2$Atilim University, Department of Mathematics,  Incek  06830, Ankara, Turkey}\\ 
{\it e-mail: ovgu.gurelyilmaz@erdogan.edu.tr, sofia.ostrovska@atilim.edu.tr, mehmet.turan@atilim.edu.tr}\\
{\it Tel: +90 312 586 8211,  Fax: +90 312 586 8091}
\end{center}

\begin{abstract}
The limit $q$-Durrmeyer operator, $D_{\infty,q},$ was introduced and its approximation properties were investigated by V. Gupta in 2008 during a study of $q$-analogues for the Bernstein-Durrmeyer operator. In the present work, this operator is investigated from a different perspective. More precisely, the growth estimates are derived for the entire functions comprising the range of $D_{\infty,q}$. The interrelation between the analytic properties of a function $f$ and the rate of growth for $D_{\infty,q}f$ are established, and the sharpness of the obtained results are demonstrated.
\end{abstract}



\section{Introduction}
The significant influence of the Bernstein polynomials on modern mathematics - not only theoretical, but also applied and computational - brought about the emergence of its numerous versions and modifications. See, for example, \cite{boehm, bustamante, lorentz}. While the Bernstein polynomials serve to approximate the continuous functions on $[0, 1]$, the Kantorovich polynomials constructed with respect to the Bernstein basis are applicable for the approximation of integrable functions. Kantorovich's breakthrough idea was further developed by Durrmeyer \cite{durrmeyer} and Derriennic \cite{derriennic}. The latter proved that the Bernstein-Durrmeyer polynomials approximate functions in $L_1[0, 1]$, and also generate self-adjoint operators in $L_2[0, 1]$.

With the increasing role of the $q$-Calculus (see, e.g. \cite{askey,castellani,chung,rima}), the $q$-analogues of various Bernstein-type operators have come to the fore. The reader is referred to \cite{bustamante, gal, phillips}. New versions of these operators, targeting a wide spectrum of various problems, are continuously coming out.

In 2008, V. Gupta \cite{gupta} introduced a simple $q$-analogue of the Bernstein-Durrmeyer operators, denoted by $D_{n,q},$ and studied its approximation properties. One of the properties that he proved was that $\{D_{n,q}\}$ converges to the limit operator $D_{\infty,q}$  in the strong operator topology on $C[0,1]$. More results on the $q$-Durrmeyer operator have been obtained in \cite{wang, mahmudov}.

In the present work, further investigation is carried out concerning the limit $q$-Bernstein-Durrmeyer operator. Distinct from the preceding studies on the subject, this paper is focused on the analytic properties that the image of $f \in C[0,1]$ possesses under the operator $D_{\infty,q}$. Here, it is proved that, for each $ f \in C[0,1]$, the function  $D_{\infty,q}f$ admits an analytic continuation from $[0,1]$ to the whole complex plane $\mathbb{C}$. The growth estimates of the entire function $D_{\infty,q}f$ are provided, along with the interconnection between the growth of $D_{\infty,q}f$ and the behaviour of $f$. The sharpness of the obtained results is demonstrated.

To present the results, let us recollect the necessary notation and definitions.
The $q$-Pochhammer symbol denotes, for each $a \in \mathbb{C}$,
\begin{align*}
(a;q)_0:=1, \qquad (a;q)_n=\prod_{j=0}^{n-1}(1-aq^j), \qquad (a;q)_\infty=\prod_{j=0}^{\infty}(1-aq^j).
\end{align*}
The Euler Identities
\begin{align}
(z;q)_\infty=\sum_{k=0}^{\infty}\frac{ (-1)^k q^{k(k-1)/2}}{ (q;q)_k}z^k, \quad \lvert q \rvert < 1, \label{euler}
\end{align}
and 
\begin{align}
\frac{1}{(z;q)_\infty}=\sum_{k=0}^{\infty} \frac{z^k}{(q;q)_k}, \quad |q|<1, \quad |z|<1, \label{euler1}
\end{align}
will be used permanently. See \cite[Chapter 10, Corollary 10.2.2]{askey}.

The $q$-integral over an interval $[0,a]$, first introduced by Thomae \cite{thomae} and later by Jackson \cite{jackson}, is defined as 
\begin{align}
\int_{0}^{a} f(t) \,d_qt:= (1-q)a \sum_{j=0}^{\infty} q^j f(aq^j).  \label{qint}
\end{align}

\begin{definition} \cite{gupta} Let $q \in (0,1)$, $f \in C[0,1]$. {\it The limit $q$-Durrmeyer operator} is defined by 
\begin{align*} 
(D_{\infty,q}f)(x):=D_{\infty,q}(f;x)= 
\begin{cases}\sum_{k=0}^{\infty} A_{\infty k}(f) p_{\infty k}(q;x),  & x \in [0,1), \\
f(1), & x= 1. 
\end{cases}
\end{align*}
where 
\begin{align}
A_{\infty k}(f):= \frac{q^{-k}}{1-q} \int_{0}^{1} f(t) p_{\infty k}(q;qt)  \,d_qt, \quad k=0,1,\ldots, \label{limqdurr}
\end{align}
and 
\begin{align}
p_{\infty k}(q;x)=\frac{(x;q)_\infty \, x^k}{(q;q)_k}, \quad k=0,\ldots . \label{pk}
\end{align}
\end{definition}
As coefficients \eqref{limqdurr} form a bounded sequence whatever $f\in C[0, 1]$ is, the function $D_{\infty, q}f$ admits an analytic continuation from $[0, 1]$ to open disc $\{z:|z| < 1\}.$
Taking into account \eqref{qint}, $A_{\infty k}(f)$ can also be expressed as 
\begin{align} 
A_{\infty k}(f)= \frac{(q;q)_\infty}{(q;q)_k} \sum_{j=0}^{\infty} \frac{f(q^j) q^{(k+1)j}}{(q;q)_j}. \label{Akf}
	\end{align}

Throughout the paper, letter C - with or without subscripts - denotes a positive constant whose specific value is of no importance. 
Subscripts, when used, indicate the dependence of $C$ on certain parameters. It should be pointed out that the same letter may stand for different values. Moreover, if $f$ is analytic in the closed disc $\Delta_r:=\{z: |z| \leqslant r\},$ the notation 
\begin{align*}
M(r;f):=\max_{z\in\Delta_r} |f(z)|
\end{align*}
will be employed. 

The article is organized as follows: In Section 2, the main outcomes are stated, while Section 3 contains the auxiliary technical lemmas. Finally, the proofs of the main results appear in Section 4.

\section{Statement of Results}
\begin{theorem} \label{ancont}
For each $f \in C[0,1]$, the function $(D_{\infty,q}f)(x)$ admits an analytic continuation from $[0,1]$ as an entire function given by 
\begin{align}
(D_{\infty,q}f)(z)=\sum_{j=0}^{\infty} \frac{f(q^j)q^j}{(q;q)_j} \sum_{n=0}^{\infty} \frac{(-1)^n q^{n(n+1)/2}}{(q;q)_n} (z;q)_{n+j}. \label{analy}
\end{align}
\end{theorem}

The proof of Theorem \ref{ancont} presented in Section \ref{sec:pm} yields, apart from \eqref{analy}, the following corollary:

\begin{corollary} \label{est}
The growth of $D_{\infty,q}f$, for each $f \in C[0,1]$, enjoys the following estimate:
\begin{align}
M(r;D_{\infty,q}f)=O((-r;q)_{\infty}), \quad r \to \infty. \label{growth}
\end{align}
\end{corollary}

It is worth pointing out that coefficients \eqref{Akf} can be viewed as the values of the function $g(z):=(qz;q)_\infty \: \rho(z)$ at points $z=q^k$, $k=0,1, \ldots$, where 
\begin{align}
\rho(z)=\sum_{j=0}^{\infty} \frac{f(q^j) q^j}{(q;q)_j}z^j. \label{ro}
\end{align}
Since $(qz;q)_\infty$ is entire and the series converges in the disc $\{z: |z|<1/q \}$ for any $f \in C[0,1]$, it follows that $g$ is analytic in that disc. Clearly, the radius of convergence for $\rho$ can be greater than $1/q$. The representation below of $D_{\infty,q}$ with the help of divided differences of $g$ is important.

\begin{theorem} \label{divi}
Given $f \in C[0,1]$, let $g(z)=(qz;q)_\infty \: \rho(z)$, where $\rho$ is defined by \eqref{ro}. Then,
\begin{align*}
(D_{\infty,q}f)(z)=\sum_{k=0}^{\infty} (-1)^k q^{k(k-1)/2} \: g[1;q; \ldots; q^k]z^k, \quad z \in \mathbb{C}.
\end{align*}
Here, $g[x_0;\ldots; x_k]$ stands for the divided difference of $g$ at the distinct nodes $x_0,\ldots, x_k$.
\end{theorem}

This representation allows us to not only refine the estimate of Corollary \ref{est}, but also establish a connection between the behaviour of $f$ and the growth of its image under $D_{\infty,q}$.

\begin{theorem} \label{estimate}
Let $R>1$ be such that $\rho$ is analytic in $\Delta_R$. Then, 
\begin{align*}
M(r;D_{\infty,q}f)=o \left( \frac {(-r;q)_{\infty}} {r^{\lambda}} \right), \quad r \to \infty,
\end{align*}
for every $\lambda<(\ln R)/\ln(1/q)$.
\end{theorem}
As a consequence of Theorem \ref{estimate}, the crude estimate \eqref{growth} can be improved. Since $\rho$ is analytic in $\{z:|z|<1/q\},$ it is possible to assume $\lambda=0$ in Theorem \ref{estimate} and obtain the following result.

\begin{corollary}\label{cor1}
For any $f \in C[0,1]$,
\begin{align*}
M(r;D_{\infty,q}f)=o((-r;q)_{\infty}), \quad r \to \infty.
\end{align*}
\end{corollary}

\begin{corollary} \label{coroestinf}
If $f(q^j)=O(q^{\alpha j})$, $j \to \infty$, for some $\alpha>0$, then 
\begin{align}
M(r;D_{\infty,q}f)=o(r^{-\lambda}(-r;q)_{\infty}), \quad r \to \infty, \label{estinf}
\end{align}
for all $\lambda< 1+ \alpha.$
\end{corollary}
Indeed, in this case, $\rho$ is analytic in $\{z:|z|<q^{-1-\alpha}\}.$
\begin{corollary} \label{cor3}
If, for every $\alpha>0$, the estimate $f(q^j)=o(q^{\alpha j})$, $j \to \infty$ holds, then, for every $\lambda \geqslant 0,$ \eqref{estinf} is true. 
\end{corollary}

The estimate in Theorem \ref{estimate} is sharp as demonstrated by the assertion below.

\begin{theorem} \label{teocntrd}
For every $\lambda > 1,$ there exists $f\in C[0, 1]$ such that 
\begin{align*}
M(r;D_{\infty,q}f)\geqslant C r^{-\lambda}(-r;q)_{\infty}, \quad r \to \infty. 
\end{align*}
\end{theorem}
Theorem \ref{estimate} and Corollaries \ref{cor1}-\ref{cor3} establish the connection between the radius of convergence for series \eqref{ro} and the rate of growth for $D_{\infty,q}f.$ In the general sense, the greater the radius is, the slower the growth becomes. Approaching the problem from a different angle, the dependence of the growth on the differentiability of $f$ at the origin is addressed in the next assertion. The statement makes it possible to obtain better estimates for $M(r;D_{\infty,q}f)$ than those guaranteed by Theorem \ref{estimate} when $f$ is differentiable at $0$ even though the series \eqref{ro} converges solely in the smallest admissible disc.

\begin{theorem} \label{teoestinf1}
Let $f$ be $m$ times differentiable at $0$ from the right. Then, 
\begin{align}
M(r;D_{\infty,q}f)=o(r^{-\lambda}(-r;q)_{\infty}), \quad r \to \infty, \label{estinf1}
\end{align}
for all $\lambda<1+m.$
\end{theorem}

\begin{corollary}
If $f$ is infinitely differentiable at $0$ from the right, then \eqref{estinf1} holds for all $\lambda > 0.$
In particular, \eqref{estinf1} is valid whenever $f$ is analytic in a neighbourhood of 0.
\end{corollary}

\section{Auxiliary Results}

In what comes next, the function $ \tau $ given by
\begin{align*}
\tau(z)= (z;q)_{\infty} \sum_{k=0}^{\infty} \frac{z^k}{(q;q)_k^{2}}, \quad |z|<1,
\end{align*}
plays a key role.

\begin{lemma}
The function $\tau$ admits an analytic continuation from the open unit disc as an entire function.
\end{lemma}

\begin{proof}
Consider 
\begin{align*}
\sum_{k=0}^{\infty} \frac{z^k}{(q;q)_k^2}=\sum_{k=0}^{\infty} \frac{z^k}{(q;q)_k} \frac{(q^{k+1};q)_\infty}{(q;q)_\infty}.
\end{align*}
By \eqref{euler}, with $z=q^{k+1},$ one has
\begin{align*}
(q^{k+1}; q)_\infty=\sum_{n=0}^{\infty} \frac{(-1)^n q^{n(n-1)/2}}{(q;q)_n} (q^{k+1})^n,
\end{align*}
whence
\begin{align*}
\sum_{k=0}^{\infty} \frac{z^k}{(q;q)_k^2}&= \frac{1}{(q;q)_\infty} \sum_{k=0}^{\infty} \frac{z^k}{(q;q)_k} \sum_{n=0}^{\infty} \frac{(-1)^n q^{n(n-1)/2} q^{(k+1)n}}{(q;q)_n} \\
&=\frac{1}{(q;q)_\infty} \sum_{n=0}^{\infty} \frac{(-1)^n q^{n(n+1)/2}}{(q;q)_n} \sum_{k=0}^{\infty} \frac{(q^n z)^k}{(q;q)_k}.
\end{align*}
By virtue of \eqref{euler1}, it follows that
\begin{align*}
\sum_{k=0}^{\infty} \frac{z^k}{(q;q)_k^2}=\frac{1}{(q;q)_\infty} \sum_{n=0}^{\infty} \frac{(-1)^n q^{n(n+1)/2}}{(q;q)_n} \frac{1}{(q^n z;q)_\infty}, \quad |z|<1.
\end{align*}
Consequently, one obtains
\begin{align}
\tau(z)&= \frac{1}{(q;q)_\infty}\sum_{n=0}^{\infty} \frac{(-1)^n q^{n(n+1)/2}}{(q;q)_n} \frac{(z;q)_\infty}{(q^n z;q)_\infty} \nonumber \\
&=\frac{1}{(q;q)_\infty}\sum_{n=0}^{\infty} \frac{(-1)^n q^{n(n+1)/2}}{(q;q)_n} (z;q)_n, \quad |z|<1. \label{tau1}
\end{align}
Now, if $z\in\Delta_R,$ then
\begin{align*}
\sum_{n=0}^{\infty}  \left| \frac{(-1)^n q^{n(n+1)/2}}{(q;q)_n} (z;q)_n \right| \leqslant \sum_{n=0}^{\infty} \frac{q^{n(n+1)/2}}{(q;q)_n} (1+R)^n < \infty.
\end{align*}
Hence, $\tau(z)$ is analytic in $\Delta_R$ for each $R>0$ and \eqref{tau1} is valid for all $z \in \mathbb{C}$. Therefore, $\tau(z)$ is an entire function.
\end{proof}

\begin{lemma} \label{growdivi}
Let $R>1$ be such that $\rho$ given by \eqref{ro} is analytic in $\{z: |z|\leqslant R \}$. Then, 
\begin{align*}
\left|g[1;q; \ldots; q^k]\right| \leqslant C q^{\lambda k}
\end{align*}
for every $ \lambda<(\ln R)/\ln(1/q)$.
\end{lemma}

\begin{proof}
It is known that (see for example, \cite [Section 2.7., p.44, Eq. (4)]{lorentz})
\begin{align*}
g[a_0; \ldots; a_k] =\frac{1}{2 \pi i} \oint_L \frac{g(\zeta)d\zeta}{(\zeta-a_0) \ldots (\zeta-a_k)},
\end{align*}
where $L$ is a positively-oriented, simple and closed curve encircling the distinct points $a_0, \ldots, a_k$ and $g$ is analytic anywhere on and inside $L$.

Therefore,
\begin{align*}
g[1;q; \ldots; q^k] = \frac{1}{2 \pi i} \oint_{|\zeta|=R} \frac{g(\zeta)d\zeta}{(\zeta-1) (\zeta-q) \ldots (\zeta-q^k)}.        
\end{align*}
Now, given $\lambda_0$ such that $0 < \lambda_0 < (\ln R)/\ln(1/q),$ that is, $1<q^{-\lambda_0}<R.$ The two cases will be considered:

\noindent{\bf Case 1.} If $q^{-\lambda_0} \leqslant R-1,$ then $g[1;q;\ldots;q^k]$ can be estimated as
\begin{align*}
\left|g[1;q; \ldots; q^k]\right| &\leqslant \frac{1}{2 \pi} \cdot \frac{M(R;g)}{(R-1) (R-q) \ldots (R-q^k)} \cdot 2 \pi R\\
& \leqslant \frac{M(R;g)R}{(R-1)^{k+1}} \leqslant 2M(R;g) q^{\lambda_0 k}.        
\end{align*}
\noindent{\bf Case 2.} If $R-1<q^{-\lambda_0} \leqslant R,$ then opt for $m_0\in \mathbb{N}_0$ such that $R-q^m>q^{-\lambda_0}$ whenever $m \geqslant m_0.$ 
Then, for $k \geqslant m_0,$ one has
\begin{align*}
\left|g[1;q; \ldots; q^k]\right| &\leqslant  \frac{M(R;g)R}{(R-1)\cdots(R-q^{m_0-1})(R-q^{m_0})\cdots (R-q^k)} \\
& \leqslant  \frac{M(R;g)R}{(R-1)\cdots(R-q^{m_0-1})} \cdot \frac{1}{(R-q^{m_0})^{k-m_0+1}} \\
& \leqslant C_{R,q,g} \frac{1}{(R-q^{m_0})^{k}} < C q^{\lambda_0 k}, \quad k\geqslant m_0.        
\end{align*} 
As a result, $\left|g[1;q; \ldots; q^k]\right| \leqslant C q^{\lambda_0 k}$ for all $k,$ possibly with a different $C.$

Combining the outcomes of the two cases yields   
$\left|g[1;q; \ldots; q^k]\right| \leqslant C q^{\lambda_0 k},$
and, in turn,  
$\left|g[1;q; \ldots; q^k]\right| \leqslant C q^{\lambda k}$ for all $\lambda \leqslant \lambda_0.$
Since $\lambda_0$ has been chosen arbitrarily, it follows that the latter inequality holds for all $\lambda < (\ln R)/\ln(1/q)$
as stated.
\end{proof}

\section{Proofs of Main Results}\label{sec:pm}

\begin{proof}[Proof of Theorem \ref{ancont}]
Using \eqref{Akf}, one obtains 
\begin{align*}
(D_{\infty,q}f)(z)&=\sum_{k=0}^{\infty} \left( \frac{(q,q)_\infty}{(q,q)_k} \sum_{j=0}^{\infty} \frac{f(q^j) q^{(k+1)j}}{(q,q)_j}\right)p_{\infty k}(q;z), \quad |z|<1.
\end{align*}
Recalling \eqref{pk} leads to
\begin{align*}
(D_{\infty,q}f)(z)&=\sum_{k=0}^{\infty} \frac{(q;q)_\infty}{(q;q)_k} \sum_{j=0}^{\infty} \frac{f(q^j)q^{(k+1)j}}{(q;q)_j} \frac{(z;q)_\infty z^k}{(q;q)_k} \\
&=(q;q)_\infty (z;q)_\infty \sum_{j=0}^{\infty} \frac{f(q^j)q^{j}}{(q;q)_j} \sum_{k=0}^{\infty} \frac{(q^jz)^k}{(q;q)_k^2} \\
&=(q;q)_\infty (z;q)_\infty \sum_{j=0}^{\infty} \frac{f(q^j)q^{j}}{(q;q)_j} \frac{\tau(q^jz)}{(q^jz;q)_\infty},  \\
&=(q;q)_\infty \sum_{j=0}^{\infty} \frac{f(q^j)q^{j}}{(q;q)_j} (z;q)_j \tau(q^jz), \quad |z|<1.
\end{align*}
By \eqref{tau1},
\begin{align*}
\tau(q^jz)&= \frac{1}{(q;q)_\infty}\sum_{n=0}^{\infty} \frac{(-1)^n q^{n(n+1)/2}}{(q;q)_n} (q^j z;q)_n,
\end{align*}
and, hence,
\begin{align}
(D_{\infty,q}f)(z)&=\sum_{j=0}^{\infty} \frac{f(q^j)q^{j}}{(q;q)_j} (z;q)_j \sum_{n=0}^{\infty} \frac{(-1)^n q^{n(n+1)/2}}{(q;q)_n} (q^j z;q)_n \nonumber \\
&=\sum_{j=0}^{\infty} \frac{f(q^j)q^{j}}{(q;q)_j} \sum_{n=0}^{\infty} \frac{(-1)^n q^{n(n+1)/2}}{(q;q)_n} (z;q)_{j+n}, \quad |z|<1. \label{dinf}
\end{align}
Since, for $R>0$ and $z\in\Delta_R$, one has $|(z;q)_{j+n}| \leqslant (-R;q)_\infty$ for all $j,n \in \mathbb{N}_0$, the series in \eqref{dinf} converges uniformly in any closed disc $\Delta_R.$ Therefore, 
\begin{align*}
\left| \sum_{n=0}^{\infty} \frac{(-1)^n q^{n(n+1)/2}}{(q;q)_n} (z;q)_{j+n}\right| \leqslant (-R;q)_\infty \sum_{n=0}^{\infty} \frac{ q^{n(n+1)/2}}{(q;q)_n} =(-R;q)_\infty (-q;q)_\infty
\end{align*} 
which implies that, when $z\in\Delta_R,$
\begin{align*}
\left|(D_{\infty,q}f)(z) \right| \leqslant (-R;q)_\infty (-q;q)_\infty \sum_{j=0}^{\infty} \frac{ \left|f(q^j) \right| q^{j}}{(q;q)_j} \leqslant \|f\|_{C[0,1]} (-R;q)_\infty \frac{(-q;q)_\infty}{(q;q)_\infty}=:C_{f,q} (-R;q)_\infty.
\end{align*}
Consequently, $(D_{\infty,q}f)(z)$ is analytic in any disc of radius $R>0$. Thus, $(D_{\infty,q}f)(z)$ is entire. This completes the proof.
\end{proof}

\begin{proof}[Proof of Theorem \ref{divi}]
Starting from \eqref{Akf}, one arrives at
\begin{align*}
A_{\infty k}(f)&=(q^{k+1};q)_\infty \sum_{j=0}^{\infty} \frac{f(q^j) q^{(k+1)j}}{(q;q)_j}=[(qz;q)_\infty \: \rho(z)]\Big|_{z=q^k}=g(q^k).
\end{align*}
Therefore,
\begin{align*} 
(D_{\infty,q}f)(z)=(z;q)_\infty \sum_{k=0}^{\infty} g(q^k) \frac{z^k}{(q;q)_k}, \quad |z|<1/q.
\end{align*}
Application of Euler's identity \eqref{euler} leads to
\begin{align*}
(D_{\infty,q}f)(z)&=\sum_{k=0}^{\infty} \sum_{j=0}^{\infty} \frac{(-1)^k q^{k(k-1)/2} g(q^j) z^{k+j} }{(q;q)_k(q;q)_j} \\
&=\sum_{k=0}^{\infty} \sum_{j=0}^{k}  \frac{(-1)^{k-j} q^{(k-j)(k-j-1)/2} g(q^j) z^{k} }{(q;q)_{k-j}(q;q)_j} \\
&=\sum_{k=0}^{\infty} (-1)^k q^{k(k-1)/2} \left(\sum_{j=0}^{k}  \frac{(-1)^{-j} g(q^j)}{q^{j(j-1)/2} (q;q)_j q^{j(k-j)} (q;q)_{k-j}}\right) z^{k}, \quad |z|<1/q.
\end{align*}
Employing \cite[p. 44, formula (3)]{lorentz} with $x_j=q^j$, one arrives at
\begin{align*}
g[1;q; \ldots; q^k]=\sum_{j=0}^{k}  \frac{(-1)^{-j} g(q^j)}{q^{j(j-1)/2} (q;q)_j q^{j(k-j)} (q;q)_{k-j}}.
\end{align*}
Therefore, formula 
\begin{align*}
(D_{\infty,q}f)(z)=\sum_{k=0}^{\infty} (-1)^k q^{k(k-1)/2} \: g[1;q; \ldots; q^k]z^k
\end{align*}
holds for $|z|<1/q$ and also in every disc where $D_{\infty,q}f$ possess an analytic continuation. Applying Theorem \ref{ancont}, one completes the proof.
\end{proof}

\begin{proof}[Proof of Theorem \ref{estimate}]
By Theorem \ref{divi},
\begin{align*}
(D_{\infty,q}f)(z)=\sum_{k=0}^{\infty} (-1)^k q^{k(k-1)/2} \: g[1;q; \ldots; q^k]z^k, \quad z \in \mathbb{C}.
\end{align*}
Select $\lambda<(\ln R)/\ln(1/q)$ and take $\mu$ such that $\lambda<\mu<(\ln R)/\ln(1/q)$. Now, the growth of $D_{\infty,q}f$ may be estimated with the help of Lemma \ref{growdivi}, which implies $|g[1;q;\ldots;q^k]| \leqslant C q^{\mu k}.$ Therefore, 
\begin{align*}
\left| (D_{\infty,q}f)(z) \right|& \leqslant C \sum_{k=0}^{\infty} q^{k(k-1)/2} \left( q^{\mu} |z|\right)^k \leqslant C \sum_{k=0}^{\infty} \frac{q^{k(k-1)/2}}{(q;q)_k} \left( q^{\mu} |z|\right)^k, 
\end{align*}
and, hence,
\begin{align*}
M(r;D_{\infty,q}f)\leqslant C (-q^\mu r;q)_\infty. 
\end{align*}
Recall \cite[formula (2.6)]{zeng} that, for $r$ large enough,
\begin{align*}
C_1 \exp \Bigl\{\frac{\ln^2r}{2 \ln(1/q)}+\frac{\ln r}{2}\Bigr\} \leqslant (-r;q)_\infty \leqslant C_2 \exp \Bigl\{\frac{\ln^2r}{2 \ln(1/q)}+\frac{\ln r}{2}\Bigr\}.
\end{align*}
Consequently, 
\begin{align}
C_1 \frac{(-r;q)_\infty }{r^\mu} \leqslant (-q^\mu r;q)_\infty \leqslant C_2 \frac{(-r;q)_\infty }{r^\mu} \label{zz2}
\end{align}
for $r$ large enough. 

As a result,
\begin{align*}
M(r;D_{\infty,q}f)&=O\left( \frac {(-r;q)_{\infty}} {r^{\mu}} \right), \quad r \to \infty, \\
&=o \left( \frac {(-r;q)_{\infty}} {r^{\lambda}} \right), \quad r \to \infty,
\end{align*}
as stated.
\end{proof}

\begin{proof}[Proof of Theorem \ref{teocntrd}]
For $\lambda>1,$ set $\alpha=q^{\lambda-1} \in (0,1)$ and $$s_j=\sum_{k=0}^j \frac{\alpha^k}{(q;q)_{j-k}}, \quad j\in\mathbb{N}_0.$$
Obviously, the sequence $\{s_j\}$ is bounded. In addition, it is increasing because, for $j\in\mathbb{N}_0,$ 
$$s_{j+1}-s_j=\sum_{k=0}^j \alpha^k \left(\frac{1}{(q;q)_{j+1-k}}-\frac{1}{(q;q)_{j-k}}\right) + \alpha^{j+1} > 0.$$
Consequently, $\{s_j\}$ converges. Now, let $f\in C[0, 1]$ be such that $f(q^j)=(q;q)_js_j.$ This is possible due to the fact that $\{(q;q)_js_j\}$ is convergent as a product of two convergent sequences. For this $f$, one has
$$
\rho(z)=\sum_{j=0}^\infty s_j (qz)^j.
$$
Evidently, $\rho$ is analytic in $\{z:|z|<1/q\}$ and 
$$
\rho(z)=\sum_{j=0}^\infty \left(\sum_{k=0}^j \frac{\alpha^k}{(q;q)_{j-k}}\right) (qz)^j
=\sum_{j=0}^\infty \frac{(qz)^j}{(q;q)_{j}} \sum_{k=0}^\infty (\alpha qz)^k = \frac{1}{(qz;q)_\infty} \cdot \frac{1}{1-\alpha q z}, \quad |z|<1/q.
$$
Hence, $g(z)=\rho(z) (qz;q)_\infty = 1/(1-\alpha q z),$ whence $g$ is analytic in $\{z:|z|<1/(\alpha q)\}.$ Plain calculations reveal: 
$$
g^{(k)}(z)= \frac{(\alpha q)^k k!}{(1-\alpha q z)^{k+1}}, \quad k \in\mathbb{N}_0.
$$
By the Intermediate Value Theorem, 
$$
g[1;q;\ldots;q^k] = \frac{g^{(k)}(\xi)}{k!},\quad \xi \in(q^k, 1).
$$
Since all $g^{(k)}(x)$ are increasing on $[0, 1],$ there holds
$$
g[1;q;\ldots;q^k] \geqslant \frac{g^{(k)}(q^k)}{k!}=\frac{(\alpha q)^k}{(1-\alpha q^{k+1})^{k+1}} \geqslant (\alpha q)^k,\quad k \in \mathbb{N}_0.
$$
As a result,
\begin{align*}
M(r; D_{\infty, q}f) &= \sum_{k=0}^\infty q^{k(k-1)/2}g[1;q;\ldots;q^k] r^k \geqslant (q;q)_\infty \sum_{k=0}^\infty \frac{q^{k(k-1)/2}}{(q;q)_k} (\alpha q r)^k
=(q;q)_\infty (-\alpha q r;q)_\infty.
\end{align*}
Writing $\alpha=q^{\lambda-1}$ and using \eqref{zz2}, one obtains
$$
M(r; D_{\infty, q}f) \geqslant C r^{-\lambda} (-r; q)_\infty, \quad r\to \infty,
$$
which completes the proof.
\end{proof}

\begin{proof}[Proof of Theorem \ref{teoestinf1}]
By Taylor's Theorem, one can write
\begin{align*}
f(x)=T_m(x)+S_m(x)
\end{align*}
where $T_m(x)$ is a polynomial of degree at most $m$ and $S_m(x)=o(x^m)$ as $x \to 0^+$. Since $D_{\infty,q}$ maps a polynomial to a polynomial of the same degree (see \cite [Remark 3]{gupta}), there holds
\begin{align*}
(D_{\infty,q}f)(z)=P_m(z)+(D_{\infty,q}S_m)(z),
\end{align*}
where $P_m(z)$ is a polynomial of degree at most $m$ and, as such,
\begin{align*}
M(r;P_m)=o(r^{-\lambda}(-r;q)_{\infty}), \quad r \to \infty, 
\end{align*}
for all $\lambda>0$. As for $M(r;D_{\infty,q}S_m)$, it can be estimated by means of Corollary \ref{coroestinf} with $\alpha=m$.
\end{proof}


\begin{thebibliography}{99}

\bibitem{askey} G.E. Andrews, R. Askey, R. Roy, 
{\it Special Functions, Encyclopedia of Mathematics and Its Applications}, 
The University Press, Cambridge, 1999, 664 pp.

\bibitem{boehm} W. Boehm, A. Müller, 
{\it On de Casteljau's algorithm}, 
Computer Aided Geometric Design {\bf 16}, (1999) 587--605.

\bibitem{bustamante} J. Bustamante, 
{\it Bernstein Operators and Their Properties}, 
Birkhäuser/Springer, Cham (2017).

\bibitem{castellani} L. Castellani, J. Wess (eds.): 
{\it Quantum Groups and Their Applications in Physics}, 
p. 652. IOS Press, (1996).

\bibitem{chung} W.S. Chung, H. Hassanabadi,
{\it The $q$-boson Algebra and su$_q$(2) Algebra Based on $q$-deformed Binary Operations,}
International Journal of Theoretical Physics {\bf 60} (6), (2021) 2102--2114.

\bibitem{derriennic} M.M. Derriennic, 
{\it Sur l'approximation de fonctions integrables sur [0,1] par des polynomes de Bernstein modifies},
J. Approx. Theory, {\bf 31}, (1981) 325--343.

\bibitem{durrmeyer}J.L. Durrmeyer, 
{\it Une formule d'inversion de la transformee de Laplace: Applications e la theorie des moments},
These de 3e cycle, Paris, (1967).

\bibitem{gal} S.G. Gal, 
{\it Approximation by Complex Bernstein and Convolution Type Operators},
World Scientific Publishing Company, 2009. 

\bibitem{gupta} V. Gupta, 
{\it Some approximation properties of $q$-Durrmeyer operators},
 Applied Mathematics and Computation, {\bf 197}(1), (2008) 172--178.

\bibitem{wang} V. Gupta, H. Wang,
{\it The rate of convergence of $q$-Durrmeyer operators for $0<q<1$},
 Math. Methods Appl. Sci., {\bf 31}(16), (2008) 1946--1955.

\bibitem{jackson}  F.H. Jackson, 
{\it On $q$-definite integrals},
 Quart. J. Pure Appl. Math., {\bf 41}, (1910) 193--203. 

\bibitem{lorentz} G.G. Lorentz, 
{\it Bernstein Polynomials}, 
Chelsea, New York, (1986).

\bibitem{mahmudov} N.I. Mahmudov,  
{\it Approximation by $q$-Durrmeyer type polynomials in compact disks in the case $q>1$},
Applied Mathematics and Computation, 237, (2014), 293--303.

\bibitem{phillips} G.M. Phillips,
{\it Interpolation and Approximation by Polynomials}, 
CMS Books Math., Springer-Verlag, 2003.

\bibitem{rima} S. Ostrovska,  
{\it The $q$-Versions of the Bernstein Operator: From Mere Analogies to Further Developments.} 
Results in Mathematics, {\bf 69} (3-4) (2016), 275--295.

\bibitem{thomae} J. Thomae, 
{\it Beitrage zur Theorie der durch die Heinsche Reihe}, 
J. Reine. Angew. Math. {\bf 70}, (1869) 258–-281.

\bibitem{zeng} J. Zeng and C. Zhang, 
{\it A $q$-analog of Newton's series, Stirling functions and Eulerian functions},
{Results in Mathematics,} {\bf 25} (1994) 370--391.

\end{thebibliography}
\end{document}